\documentclass[12pt]{article}
\usepackage{amsfonts}
\usepackage{amssymb}
\usepackage{amsmath}
\usepackage{latexsym}
\usepackage{epsfig}

\usepackage{xcolor}

\usepackage[T2A]{fontenc}
\usepackage[utf8]{inputenc}

\newcommand{\im}{\hbox{\rm Im}\,}

\newcommand{\spec}{\hbox{\rm sp}\,}

\newcommand{\re}{\hbox{\rm Re}\,}

\newcommand{\Range}{\mathop{\rm Range}\nolimits}

\newcommand{\R}{{\Bbb R}}

\newcommand{\Aa}{{\mathcal{A}}}

\newcommand{\Compl}{{\Bbb C}}

\def\nin{\relax\hbox{$/\kern-.7em{\rm \in\,}$}}

\newenvironment{proof}{\par
\noindent
{\sc Proof. }}{ \hfill $\square$
\vspace{0.3cm}
\par}

\newtheorem{theorem}{Theorem}[section]
\newtheorem{lemma}[theorem]{Lemma}
\newtheorem{proposition}[theorem]{Proposition}
\newtheorem{corollary}[theorem]{Corollary}
\newtheorem{remark}[theorem]{Remark}

\begin{document}
	\fontsize{15.8pt}{21.06pt}\rm
\date{}

\title{Mixed Fourier norm spaces of analytic functions on the upper half-plane and Toeplitz operators\\ }
\author{\vspace{.3cm}
{\bf Zhirayr Avetisyan, Alexey Karapetyants, and Irina Smirnova}}

\maketitle
\begin{abstract}
We introduce and study weighted spaces of functions with mixed norm on the upper half-plane, defined in terms of Fourier transform. We give a characterization of analytic functions within these spaces, and in particular, we provide an analog of the Paley–Wiener theorem in this setting. As an application, we consider Toeplitz operators with vertical symbols in these new spaces.
\end{abstract}


\section{Introduction}
Spaces of analytic $p$ - integrable functions on the unit disc or half-plane have been a subject of extensive interest in analysis, and the results concerning the functional properties of these spaces, and the action of certain operators in these spaces, are nowadays classical and form an essential part of a number of textbooks on the theory of functions and the theory of operators in complex analysis. Without claiming to present a complete overview, we refer to the books \cite{DjrbashianShamoyan-book1988, Bergman-book1970, DS-book, Vasilevski-book, Zhu-book, Zhu-book2}.

Weighted classical spaces of analytic functions in the disc and on the half-plane, with a special choice of weight consistent with the geometry of the space, have also received special attention in connection with the study of special classes of Toeplitz operators in such spaces, see the series of papers \cite{Vasilevski-Grudski-Karapetyants-SEMR-2006,Grudsky-Karapetyants-Vasilevski-JOpTh-2003,Grudsky-Karapetyants-Vasilevski-IEOT-2003,Grudsky-Karapetyants-Vasilevski-BolSocMathMex-2004,Grudsky-Karapetyants-Vasilevski-JOpTh-2004} and the book \cite{Vasilevski-book}.

One of the further generalizations of the theory of analytic functions in the disс and on the half-plane is the consideration of the function spaces equipped with (weighted) mixed norm. First of all, the integral mixed norm is studied in the literature; we refer to the book \cite{ArsenovicEtAl} and some earlier articles \cite{Gu1992,Gadbois-1988} (see also \cite{ArsenovicEtAl} for more references).

From the point of view of complex analysis, the introduction of a mixed norm is more than justified by the fact that there is often a need to measure the behavior of an analytic function only when approaching a boundary, say, in an orthogonal direction, and also to measure this behavior according to some special norm (that may differ from the norm used in another direction).

An approach is also developing, in which a special mixed norm is introduced, so that it is not the functions themselves that are measured, but rather their Fourier images with respect to one of the variables, which is actually "responsible" for analyticity. This situation is most transparent in the unit disc. It was implemented in a number of papers, and a variety of norms were considered, including the norms of the variable Lebesgue spaces, those of Morrey, Orlicz spaces, various combinations of such norms, and various effects arising from such a diversity of norms were shown - see more details in the papers \cite{KS-2018,KS-2017,KS-2016-1,KS-2016-2,KSmi-CVEE,KSmi-IzvVuz, K-Grand-2021}.

The unit disc situation has a certain advantage in view of the expansion of the analytic function into a Taylor series. It is precisely because of this that it becomes possible to characterize uniquely and explicitly the analytic functions from such new spaces defined through the Fourier coefficients; for more details, see the above-mentioned sources.

Until recently, such a study was not clear in the case of a half-plane and a mixed norm associated with Cartesian or polar coordinates. The question is: what is a constructive description of analytic functions in the upper half-plane whose Fourier transforms with respect to the real part of the argument are taken from some spaces with a mixed norm? Recently, this problem was solved in a general form in \cite{AvetisyanKarapetyants-2024}, which opened the way to the study of a wide class of function spaces with a mixed norm, similar to that previously carried out in the case of the unit disc in the above-mentioned papers.

In the present paper, we consider the situation when the Fourier image of an analytic function with respect to the first variable (i.e., the real part of the complex argument) belongs to a space of functions with a weighted mixed integral norm - see below for more details and exact definitions.

Naturally, in this particular situation, the results of the paper \cite{AvetisyanKarapetyants-2024} are clarified and reformulated in more transparent and explicit forms, and more precise statements can be formulated and proved.

Next, we provide a representation for our analytic mixed norm spaces $\Aa_{\lambda}^{q,p}(\Pi)$ using some isomorphisms. More precisely, we establish an isometric isomorphism between $\Aa_{\lambda}^{q,p}(\Pi)$ and  $ L^{q}(\R_{+})$ via certain explicit Fourier-type transforms. This result is an analogue (a generalization) of the Paley-Wiener theorem for our situation.

Finally, we apply the obtained results to the study of Toeplitz operators with vertical symbols in $\Aa_{\lambda}^{q,p}(\Pi)$ . This illustrates the utility and importance of our approach. In particular, the ideas of \cite{KL-2020} can be developed in this general case as well.

Of course, one can move further in this direction by studying algebras of such Toeplitz operators and further properties of operators and spaces $\Aa_{\lambda}^{q,p}(\Pi)$, similar to the classical situation, see, e.g. \cite{Grudsky-Karapetyants-Vasilevski-JOpTh-2004}.

A similar study can be made for the case of mixed norm spaces constructed with the use of polar coordinates in a half-plane, but this is a question for another investigation.

\section{Preliminaries}
The Fourier transform $F: L^{2}(\R)\longrightarrow L^{2}(\R)$ is defined on functions in $L^{1}(\R)$ by
\[
(F f)(u)=\frac{1}{\sqrt{2\pi}}\int_{\R}e^{-iux}f(x)dx,\ \
f\in L^{1}(\R)
\]
and then extends in the usual way to an isomorphism in $L^2(\mathbb{R})$ by a continuous extension from the dense subset $L^1(\mathbb{R})\cap L^2(\mathbb{R}).$

The inverse Fourier transform has the form: \[ (F^{-1}f)(x)=\frac{1}{\sqrt{2\pi}}\int_{\R}e^{iux}f(u) du,\,\,\, f\in L^{1}(\R) .
\] Let $\Pi$ denote the (open) upper half-plane in the complex space $\mathbb{C}.$ We identify $\Pi = \mathbb{R}\times \mathbb{R}_+$ by writing $z=x+iy,$ $z\in\Pi,$ $x\in \mathbb{R}$ and $y\in\mathbb{R}_+=(0,+\infty)$, with the corresponding understanding of integrals, measures and so on in what follows.

Let $L^{2}_{\lambda}(\Pi),$ $\lambda>-1,$ denote the space of functions $f$ measurable on the upper half-plane $\Pi$ for which the norm
\begin{equation*}
\|f\|_{L^2_{\lambda}(\Pi)}=\left(\int_{\R}\int_{\R_+} |f(x,y)|^2(\lambda+1)(2y)^\lambda dy\frac{1}{\pi}\:dx\right)^{\frac{1}{2}}
\end{equation*}
is finite. The weighted Bergman space $\mathcal{A}^{2}_{\lambda}(\Pi)$ is a closed subspace of the weighted space $L^{2}_{\lambda}(\Pi),$ consisting of functions analytic in $\Pi$ (see \cite{Zhu-book, Zhu-book2, DS-book}).

Orthogonal (weighted) Bergman projection $B_{\Pi}^{(\lambda)}$ of the space $L^{2}_{\lambda}(\Pi)$ onto $\Aa^{2}_{\lambda} (\Pi)$ has the form \begin{equation*} B^{(\lambda)}_{\Pi}f(z)=\int_{\Pi}f(w)K_{\lambda}(z,w )d\mu_{\lambda}(w), \end{equation*} where $$d\mu_\lambda(z)= \frac{1}{\pi}(\lambda+1)(2y)^{ \lambda}dx dy,\,\,\,\, z=x+iy \in \Pi, $$ and
\begin{equation*}
K_{\lambda}(z,w)=-i^\lambda\frac{1}{(z-\overline{w})^{2+\lambda}},
\end {equation*}
is the Bergman (weighted) kernel for $\Aa^{2}_{\lambda} (\Pi).$
Despite the fact that the Bergman projection is initially defined on
$\Aa^{2}_{\lambda} (\Pi),$ the above integral representation for
$B^{(\lambda)}_{\Pi}$ defines $B^{(\lambda)}_{\Pi}f$ for a wider class of functions (see the above references).

\section{General definition of mixed norm space ${\mathcal{L}}^{q;X}({\Pi})$ .}\label{GeneralSection} \setcounter{equation}{0}
By the symbol $C^\infty_c(\mathbb{R}_+)$ we denote the class of infinitely differentiable functions with compact support in $\mathbb{R}_+ .$

Let $X(\mathbb{R}_+)$ be a complex Banach space of distributions on $\mathbb{R}_+$ uniformly embedded into the space of distributions $C^\infty_c(\mathbb{R}_+)'$, i.e.,
$$
|v(h)|\le\wp(h)\cdot\|v\|_{X(\mathbb{R}_+)},\quad\forall h\in C^\infty_c(\mathbb{R}_+),\quad\forall v\in X(\mathbb{R}_+)\subset C^\infty_c(\mathbb{R}_+)',
$$
where $\wp$ is any continuous function $C^\infty_c(\mathbb{R}_+)\to[0,+\infty)$ with $\wp(0)=0$.

Note that this condition is easily fulfilled for usual spaces of locally Lebesgue-integrable functions thanks to H\"older's inequality.

Suppose that the norm ${\|\cdot\|}_{X(\mathbb{R}_+)}$ is such that the map
\begin{eqnarray}\label{measurable_test}
\mathbb{R}\ni\xi\mapsto\frac{\Phi_\xi}{\|
\Phi_\xi\|_{X(\mathbb{R}_+)}}\in C^\infty_c(\mathbb{R}_+)',\\
\Phi_\xi(y)\doteq e^{-\xi y},\quad\forall y\in\mathbb{R}_+,\quad\forall\xi\in\mathbb{R},\nonumber
\end{eqnarray}
equals almost everywhere a Bochner-measurable function $\mathbb{R}\to X(\mathbb{R}_+)$.

For instance, for $X(\mathbb{R}_+)=L^p(\mathbb{R}_+)$ with $p\in[1,+\infty)$ the above map is an explicitly computable continuous function (with the convention that $\frac1{+\infty}=0$).

Consider first the normed space
$$
L^{q;X}(\Pi)\doteq\left\{\varphi\in L^1_\mathrm{loc}(\mathbb{R},X(\mathbb{R}_+))\,\vline\quad\|\varphi(\cdot)\|_{X(\mathbb{R}_+)}\in L^q(\mathbb{R})\right\},
$$
$$
\|\varphi\|_{L^{q;X}({\Pi})}\doteq\|\|\varphi(\cdot)\|_{X(\mathbb{R}_+)}\|_{L^q(\mathbb{R})},\quad1\le q<\infty.
$$
We will interpret $\varphi\in L^{q;X}(\Pi)$ as distributions $\varphi(x,y)$ on $\Pi$ (see below), such that the map $x\mapsto\varphi(x,\cdot)$ is $\mathbb{R}\to X(\mathbb{R}_+)$, and the map $x\mapsto\|\varphi(x,\cdot)\|_{X(\mathbb{R}_+)}$ is from $L^q(\mathbb{R})$.

That $$L^{q;X}({\Pi})\subset L^1_\mathrm{loc}(\mathbb{R},X(\mathbb{R}_+))$$ is a vector subspace and $\|\cdot\|_{L^{q;X}({\Pi})}$ is indeed a norm, follows from Lemma 4 in \cite{AvetisyanKarapetyants-2024}. Moreover, by Lemma 5 of \cite{AvetisyanKarapetyants-2024}, we have a continuous embedding
\begin{equation}
L^{q;X}({\Pi})\hookrightarrow\mathcal{D}_t(\Pi)'\label{LqXDtEmbedEq}
\end{equation}
into the space of distributions $\mathcal{D}_t(\Pi)'$, which are tempered in the first variable. More precisely,
$$
\mathcal{D}_t(\Pi)=\mathcal{S}(\mathbb{R})\,\widehat{\otimes}\,C^\infty_c(\mathbb{R}_+)
$$
is the tensor product of two nuclear spaces: the space of Schwartz functions on $\mathbb{R}$ and the space of test functions on $\mathbb{R}_+$. In simpler terms, $\mathcal{D}_t(\Pi)$ is the space of smooth functions which are of rapid decay in the variable $x$ and of compact support in the variable $y$. Correspondingly, $\mathcal{D}_t(\Pi)'$ is the dual space of distributions which are tempered in the variable $x$.

Introduce the mixed-Fourier-norm space ${\mathcal{L}}^{q;X}({\Pi}),$ $1\leqslant q <\infty,$ as the space of distributions $f\in\mathcal{D}_t(\Pi)'$ such that the
(distributional) Fourier transform in the first variable $\widehat{f}_{x\to\xi}(\xi,y)$, considered as a map $\xi\mapsto\widehat{f}_{x\to\xi}(\xi,\cdot)$, is a locally Bochner-integrable function $${\widehat{f}_{x\to\xi}:\mathbb{R}\to X(\mathbb{R}_+)},$$ and
the following norm is finite:
\begin{equation}\label{eq: def Lap}
\|f \|_{{\mathcal{L}}^{q;X}({\Pi})}= \left( \int_{\mathbb{R}} \bigg{\|} \frac{1}{\sqrt{\pi}}\widehat{f}_{x\to\xi}(\xi,\cdot)\bigg{\|}_{X(\mathbb{R}_+)}^q d\xi
\right)^{\frac{1}{q}}=\left\|\frac{1}{\sqrt{\pi}}\widehat{f}_{x\to\xi}\right\|_{L^{q;X}(\Pi)}.
\end{equation}
In other words,
$$
{\mathcal{L}}^{q;X}({\Pi})=\left\{f\in\mathcal{D}_t(\Pi)'\,\vline\quad\widehat{f}_{x\to\xi}\in L^{q;X}(\Pi)\right\},
$$
see \cite{AvetisyanKarapetyants-2024} for details.
\begin{proposition}\label{prop:II}
The space ${\mathcal{L}}^{q;X}({\Pi})$ is a normed vector space, continuously embedded in $\mathcal{D}_t(\Pi)'$, and isometrically isomorphic to $L^{q;X}(\Pi)$ via the Fourier transform in the first variable: $$\frac{1}{\sqrt{\pi}} F\,\otimes\,I:\mathcal{L}^{q;X}({\Pi})\to L^{q;X}(\Pi).$$
\end{proposition}
\begin{proof}
Indeed, that it is an isometry is clear from (\ref{eq: def Lap}), and the surjectivity follows from (\ref{LqXDtEmbedEq}).
\end{proof}

Note that the factor $\frac{1}{\sqrt{\pi}}$ in the definition above of the operator $\frac{1}{\sqrt{\pi}} F\,\otimes\,I$ has a purely technical function and is intended to ensure that our results coincide with the classical ones in the case of $X(\mathbb{R}_+) = L^2(\mathbb{R}_+)$ and $q=2$.

\section{Mixed norm spaces ${\mathcal{L}}_\lambda^{q,p}({\Pi}),$ $  L^{q,p}_\lambda(\Pi),$ $\Aa^{q,p}_{1,\lambda}(\Pi)$ and $\Aa^{q,p}_{\lambda}(\Pi)$ }

Fix a number $\lambda>-1$ as before, and consider the density
$$
d\nu_\lambda(y)=(\lambda+1)(2y)^\lambda dy
$$
on $\mathbb{R}_+$.
Choose
$$
X(\mathbb{R}_+)\doteq L^p(\mathbb{R}_+,\nu_\lambda),\,\,\, 1\leqslant p<\infty,
$$
and denote
$$
{{\mathcal{L}}_\lambda^{q,p}({\Pi})\doteq{\mathcal{L}}^{q;X}({\Pi})},\quad L^{q,p}_\lambda(\Pi)\doteq L^{q;X}(\Pi),
$$
in the sense of the definitions in Section \ref{GeneralSection}.
Hence,
\begin{eqnarray*}
L^{q,p}_\lambda(\Pi) && \doteq\bigg{\{} \varphi=\varphi(x,y)\in L^1_\mathrm{loc}(\mathbb{R},L^p(\mathbb{R}_+,\nu_\lambda))\,\vline\\ && \|\varphi(x,\cdot)\|_{L^p(\mathbb{R}_+,\nu_\lambda)}\in L^q(\mathbb{R})\bigg{\}},
\end{eqnarray*}
$$
\|\varphi\|_{L^{q,p}_\lambda(\Pi)}\doteq\|\|\varphi\|_{L^p(\mathbb{R}_+,\nu_\lambda)}\|_{L^q(\mathbb{R})},\quad1\le q<\infty.
$$
We interpret $\varphi\in L^{q,p}_\lambda(\Pi)$ as distributions $\varphi(x,y)$ on $\Pi$ given by maps $x\mapsto\varphi(x,\cdot)$ as $\mathbb{R}\to L^p(\mathbb{R}_+,\nu_\lambda)$, such that the maps $$x\mapsto\|\varphi(x,\cdot)\|_{L^p(\mathbb{R}_+,\nu_\lambda)}$$ are in $L^q(\mathbb{R})$.

Note that the conditions imposed in Section \ref{GeneralSection} on the space $X(\mathbb{R}_+)$ are satisfied by $X(\mathbb{R}_+) = L^p(\mathbb{R}_+,{\nu_\lambda}).$ In particular, the map from formula (\ref{measurable_test}) becomes
\begin{equation}\label{measurable_test_Lp}
\mathbb{R}\ni\xi\mapsto\begin{cases}
\frac{\xi^{\frac{\lambda+1}{p}}\Phi_\xi}{\|\Phi_1\|_{L^p(\mathbb{R}_+,{\nu_\lambda})}} & \mbox{if}\quad\xi>0 \\
0 \quad \mbox{else}
\end{cases}\in L^p(\mathbb{R}_+,{\nu_\lambda}),
\end{equation}
and is continuous. Thus, the choice $X(\mathbb{R}_+) = L^p(\mathbb{R}_+,{\nu_\lambda})$ is legitimate, and all conclusions drawn in Section \ref{GeneralSection} are valid for the mixed-Fourier-norm space $\mathcal{L}_\lambda^{q,p}({\Pi})$.

We note that $$\|\Phi_1\|_{L^p(\mathbb{R}_+,\nu_\lambda)}=\left(\frac{2^\lambda \Gamma (\lambda+2)}{p^{\lambda+1}}\right)^{\frac{1}{p}}.$$

Let $$\Aa^{q,p}_{\lambda}(\Pi)\doteq\mathrm{Hol}(\Pi)\cap{\mathcal{L}}_\lambda^{q,p}({\Pi})$$ be the set of those elements of ${\mathcal{L}}_\lambda^{q,p}({\Pi})$ which are analytic in $\Pi$.

\begin{lemma}\label{CompletenessLemma}
The normed subspace $\Aa^{q,p}_{\lambda}(\Pi)\subset{\mathcal{L}}_\lambda^{q,p}({\Pi})$ is complete.
\end{lemma}
\begin{proof} Note that $L^q(\mathbb{R})$ satisfies conditions (19)-(22) in \cite{AvetisyanKarapetyants-2024}, and the map
$$
\mathbb{R}\ni\xi\mapsto\begin{cases}
\frac{\xi^{\frac{\lambda+1}{p}}}{\|{\Phi_1}\|_{L^p(\mathbb{R}_+,{\nu_\lambda})}} & \mbox{if}\quad\xi>0 \\
0 \quad \mbox{else}
\end{cases}
$$
is locally bounded. Thus, by Proposition 5 in \cite{AvetisyanKarapetyants-2024}, the weighted space $$\Xi(\hat G,\rho)=L^q(\mathbb{R},\varrho),$$ given by the weight
$$
\varrho(\xi)\doteq\|\Phi_\xi\|_{L^p(\mathbb{R}_+,\nu_\lambda)}^q=\begin{cases}
\xi^{-\frac{\lambda+1}{p}q}\left(\frac{2^\lambda \Gamma (\lambda+2)}{p^{\lambda+1}}\right)^{\frac{q}{p}} & \mbox{for}\quad\xi>0,\\
+\infty & \mbox{for}\quad\xi\le0,
\end{cases}
$$
is complete (of course, the completeness of weighted Lebesgue spaces is known independently).

Note that the weight $\varrho$ is not almost everywhere finite, and takes values in $[0,+\infty]$. It should be understood that elements of $L^q(\mathbb{R},\varrho)$ are almost everywhere zero on $(-\infty,0)$.

Moreover, by Proposition 6 in \cite{AvetisyanKarapetyants-2024}, the Fourier transform in the first variable
$$
\frac{1}{\sqrt{\pi}}\,F\,\otimes\,I:\Aa^{q,p}_{\lambda}(\Pi)\to\Xi(\hat G,\rho)=L^q(\mathbb{R},\varrho)
$$
is an isometric isomorphism. Therefore, as mentioned in Remark 8 in \cite{AvetisyanKarapetyants-2024}, the space $\Aa^{q,p}_{\lambda}(\Pi)$ is complete.
\end{proof}

Define the operator $$U_{1,\lambda}:{\mathcal{L}}_\lambda^{q,p}({\Pi})\to L^{q,p}_\lambda(\Pi)$$ as
$$
U_{1,\lambda}f\doteq\frac1{\sqrt\pi}\widehat{f}_{x\to\xi},\quad\forall f\in{\mathcal{L}}_\lambda^{q,p}({\Pi}).
$$
So by definition
$$\|f\|_{{\mathcal{L}}_\lambda^{q;p}({\Pi})}=\|U_{1,\lambda} f \|_{L^{q,p}_\lambda(\Pi)}.$$ Note that (see Proposition \ref{prop:II}) the operator $U_{1,\lambda}$ defined in this way is an isometric isomorphism between normed vector spaces ${\mathcal{L}}_\lambda^{q,p}({\Pi})$ and $L^{q,p}_\lambda(\Pi)$.

Note that for $p=q=2$, $U_{1,\lambda}$ is unitary as an operator acting between: $$ U_{1,\lambda} :L_{\lambda}^{2}(\Pi )\longrightarrow L^{2}(\R, L^{2}(\R_{+},\nu_\lambda)) . $$

The space $\Aa^{q,p}_{1,\lambda}( \Pi)$ is defined as the subspace of functions $\varphi=\varphi(x,y)$, $\varphi\in L^{q,p}_\lambda(\Pi),$ for which

\begin{equation}\label{eq:DEQ}
U_{1,\lambda}\frac{\partial}{\partial \overline{z}}U_{1,\lambda}^{-1}\varphi
=\frac{i}{2}\left(x+\frac{\partial}{\partial y}\right)\varphi =0
\end{equation}
in the weak ($\mathcal{D}_t(\Pi)' $-distributional) sense.

\begin{remark}
Note that
$$\Aa^{2,2}_{\lambda}(\Pi)=\Aa^{2}_{\lambda}(\Pi),\,\,\,\,\,\,
\Aa^{2,2}_{1,\lambda}(\Pi)=U_{1,\lambda}(\Aa^{2}_{\lambda}(\Pi)).$$
\end{remark}
It is easy to check that the functions $\varphi$ satisfying (\ref{eq:DEQ}) have the form
\begin{equation*}
		\varphi(z)=\varphi(x,y)=\psi(x)e^{-xy},
\end{equation*}
where $\psi\in L^1_\mathrm{loc}(\mathbb{R})$ (see Lemma 9 in \cite{AvetisyanKarapetyants-2024}). Note that the function $\varphi$ must belong to the space $$L^q(\mathbb{R},L^p(\mathbb{R}_+,\nu_\lambda)),$$
therefore $\Aa^{q,p}_{1,\lambda}
(\Pi)$ is the subspace in $L^q(\mathbb{R},L^p(\mathbb{R}_+,\nu_\lambda))$ of functions of the form
\begin{equation*}
\varphi(x,y)=\chi_{+}(x)\theta_{\lambda,p}(x)f(x)e^{-xy},\quad f\in L^{q}(\R),
\end{equation*}
where $\chi_{+}$ is the characteristic function of the semiaxis $\R_{+},$ and
\begin{equation}
\theta_{\lambda,p}(x)=\left(\int_{\R_+}2^{\lambda}(\lambda+1)\;e^{-pxy}y^{\lambda}
dy\right)^{-\frac{1}{p}}=\left(\frac{p^{\lambda+1}x^{\lambda+1}}{2^{\lambda}\Gamma
(\lambda+2)}\right)^{\frac{1}{p}}\ ,\label{spes fu 2}
\end{equation}
for $x\geq 0$. Moreover, $$\|\varphi\|_{\Aa^{q,p}_{1,\lambda}(\Pi)}=\|f\|_{L^q(\R_+)}.$$
Indeed, straightforward calculations give
\begin{eqnarray*}
&&\|\varphi\|_{\Aa^{q,p}_{1,\lambda}(\Pi)}\\ &&=\left(\int_{\R}\left(\int_{\R_+}
\left|\chi_+(x)\theta_{\lambda,p}(x)f(x)e^{-xy}\right|^p d\nu_\lambda(y)
\right)^{\frac{q}{p}}dx\right)^{\frac{1}{q}}\\&&=\left(\int_{\R_+}|f(x)|^q\left(
\frac{p^{\lambda+1}x^{\lambda+1}}{\Gamma(\lambda+1)} \int_{\R_+} e^{-pxy} y^\lambda
dy\right)^{\frac{q}{p}}dx\right)^{\frac{1}{q}}=\|f\|_{L^q(\R_+)}.
\end{eqnarray*}

Having described the space $\Aa^{q,p}_{1,\lambda}(\Pi)$ in this way, we finally notice that it isometrically coincides with our original space $\Aa^{q,p}_{\lambda }(\Pi).$ Namely, the following result is true: \begin{theorem}(\cite{AvetisyanKarapetyants-2024}) The restriction of the isometric isomorphism \begin{equation*} {\mathcal{L}}^{q,p}_ \lambda({\Pi})\ni f\overset{U_{1,\lambda}}{\longmapsto}\frac{1}{\sqrt{\pi}}\widehat{f}_{x\to\xi}\in L^{q,p}_\lambda(\Pi) \end{equation*} to the space $\Aa^{q,p}_{\lambda}(\Pi)$ is an isometric isomorphism between $\Aa^{q,p}_{\lambda}(\Pi)$ and $\Aa^{q,p}_{1,\lambda}(\Pi),$
$$\|f\|_{{\mathcal{L}}^{q,p}_\lambda({\Pi})}=\|U_{1,\lambda} f \|_{L^{q,p}_\lambda(\Pi)} . $$
\end{theorem}
\begin{proof} Corollary 4 in \cite{AvetisyanKarapetyants-2024} describes the isomorphism between $\Aa^{q,p}_{1,\lambda}(\Pi)$ and the weighted space $\Xi(\hat G,\rho)=L^q(\mathbb{R},\varrho)$. Then the statement immediately follows from Proposition 6 in \cite{AvetisyanKarapetyants-2024} (see the proof of Lemma \ref{CompletenessLemma} above).
\end{proof}

\section{A Paley-Wiener-type theorem.}
 Introduce the operator
\begin{equation*}
U_{2,\lambda}:L^{q}(\R, L^{p}(\R_{+},\nu_\lambda))\longrightarrow
L^{q}(\R,L^{p}(\R_{+}))
\end{equation*}
as follows:
\begin{equation}
(U_{2,\lambda}\varphi)(x,y)=\frac{1}{\theta_{\lambda,p}(\mid x \mid)}\;e^{-\frac{y}{p}+\mid
x \mid \beta _{\lambda}(\mid x \mid ,y)}\varphi(x, \beta_{\lambda}(\mid x \mid,y)),\label{U22}
\end{equation}
where for any fixed $x>0$ the function $y\rightarrow \beta_{\lambda}(x,y)$ is the inverse to the
function
\begin{equation*}
\psi_{\lambda}(x,t)=\ln\left(\frac{\Gamma(\lambda+1)}{\Gamma(\lambda+1, pxt )}
\right),\quad\forall t>0,
\end{equation*}
that is,
$$\beta_{\lambda}(x,\psi_{\lambda}(x,t))=t,\quad\forall x>0,\quad\forall t>0.$$ Here $\Gamma(a,b)$ is an incomplete gamma function (see, for example, \cite{GradshteinRyzhik1980}).

Let us show that the function $\psi_{\lambda}$ indeed has the form as defined above based on the condition that the operator $U_{2,\lambda}$ is an isometry. Let $x>0,$ set
\begin{equation*}
(U_{2,\lambda}\varphi)(x,y)=\alpha_{\lambda}(x,y)\varphi(x, \beta_{\lambda}(
x ,y)),
\end{equation*}
where the function $\alpha_{\lambda}$ will be chosen later from the condition of the coincidence of norms (the isometry condition):
\[
\|U_{2,\lambda}\varphi\|_{L^{q}(\R,L^{p}(\R_{+}))}=
\|\varphi\|_{L^{q}(\R, L^{p}(\R_{+},\nu_\lambda))} .
\]
Elementary calculations under the integral sign show that this is equivalent to the following relation,
\[
\alpha_{\lambda}^{p}(x,\psi_{\lambda}(x,t))\ \frac{\partial}{\partial t}\psi_{\lambda}
(x,t)=(\lambda+1)(2t)^{\lambda}.
\]
But on the other hand,
\[
U_{2,\lambda}\theta_{\lambda,p}(x)e^{-xy}=e^{-\frac{y}{p}}.
\]
Therefore, we obtain
\begin{eqnarray*}
&&\alpha_{\lambda}(x,y)\theta_{\lambda,p}(x)e^{-x\beta_{\lambda}(x,y)}=e^{-\frac{y}{p}},\\
&&\alpha_{\lambda}(x,\psi_{\lambda}(x,t))=\theta^{-1}_{\lambda,p}(x)e^{xt- \frac{\psi_{\lambda}(x,t)}{p}}.
\end{eqnarray*}
Hence,
\[
\frac{1}{\theta_{\lambda,p}^{p}(x)}e^{pxt-\psi_{\lambda}(x,t)}\frac{\partial}{\partial t}\psi_{\lambda}(x,t)=(\lambda+1)(2t)^{\lambda}.
\]
Rewriting the last equality and further integrating it, we get
\begin{equation*}
\int_{t}^{\infty} e^{-\psi_{\lambda}(x,\tau)}\frac{\partial}{\partial \tau}\psi
_{\lambda}(x,\tau)d\tau =\theta_{\lambda,p}^{p}(x)(\lambda+1)\int_{t}^{\infty}
(2\tau)^{\lambda}e^{-px\tau}d\tau ,
\end{equation*}
thus we obtain what was required to be proven:
\begin{eqnarray}
\label{eq: exp 2} e^{-\psi_{\lambda}(x,t)}=\frac{x^{\lambda+1}}{\Gamma(\lambda+1)}\int_t^{\infty}
e^{-p\tau x}(p\tau)^{\lambda} d(p\tau)=\frac{\Gamma(\lambda+1,pxt)}{\Gamma(\lambda+1)}.
\end{eqnarray}

The inverse operator
\begin{equation*}
U^{-1}_{2,\lambda}:L^{q}(\R,L^{p}(\R_{+}))\longrightarrow L^{q}(\R, L^{p}(\R_{+},\nu_\lambda))
\end{equation*}
acts as follows:
\begin{equation*}
(U^{-1}_{2,\lambda}\varphi)(x,y)=\theta_{\lambda,p}(|x|)e^{\frac{1}{p}\psi_{\lambda}
(|x|,y)-|x|y}\varphi(x,\psi_{\lambda}(|x|,y)).
	\end{equation*}
For any function $f\in L^{q}(\R)$ we have
$$
\chi_{+}(x)\theta_{\lambda,p}(x)f(x)e^{-xy}\overset{U_{2,\lambda}}{\longmapsto} \chi_{+}(x)f(x)e^{-\frac{y}{p}} .
$$
Therefore, the image
$$\Aa^{q,p}_{2,\lambda}(\Pi)=U_{2,\lambda}(\Aa^{q,p}_{1,\lambda}(\Pi))$$
is defined as the subspace in $L^{q}(\R,L^{p}(\R_{+}))$ of functions of the form
\begin{equation*}
\psi(x,y)=\chi_{+}(x)f(x)e^{-\frac{y}{p}},\ \ \ f\in L^{q}(\R)
\end{equation*}
and coincides with the space $L^{q}(\R_{+},L^{p}_{0}(\R_{+}))$, where $L^{p}_{0}(\R_{+})$ is a one-dimensional subspace of $L^{p}(\R_{+}),$ generated by the element $\ell_{0,p}(y)=e^{-\frac{y}{p}}.$

Let $P_{0}$ denote the one-dimensional projection of the space $L^{p}(\R_{+})$ onto $L_{0}^{p} (\R_{+}):$
\[
(P_{0}\psi)(y)=\ell_{0,p}(y)\left.\int_{\R_+}\psi(v)e^{-\frac{v}{p}}dv. \right .
\]
Based on the above, the following theorems hold.
\begin{theorem}
\label{th:parabolic-decomp}Let $\lambda\in(-1,+\infty).$ The operator $U_{2,\lambda}$ is an isometric isomorphism of the space $$L^{q,p}_\lambda(\Pi)=L^{q}(\R, L^{p}(\R_{+},\nu_\lambda))$$ onto
$$L^{q,p}(\Pi)= L^{q}(\R,L^{p}(\R_{+})),$$ such that
the space $\Aa^{q,p}_{1,\lambda}(\Pi)$ is mapped onto
\[
L^{q}(\R_{+},L^{p}_{0}(\R_{+})).
\]
\end{theorem}

\begin{theorem}
\label{th:parabolic-decomp} Let $\lambda\in(-1,+\infty).$ The operator $U_{\lambda}=U_{2,\lambda}U_{1,\lambda}$ is an isometric isomorphism of the space $${\mathcal{L}}^{q,p}({\Pi})$$ onto
$$L^{q,p}(\Pi)= L^{q}(\R,L^{p}(\R_{+})),$$ such that
\begin{enumerate}
\item
The Bergman space $\Aa^{q,p}_{\lambda}(\Pi)$ is mapped onto
\[
L^{q}(\R_{+},L^{p}_{0}(\R_{+})).
\] \item The Bergman projector $B_{\Pi}^{\lambda}$ is equivalent to the projector $\chi_{+}I\otimes P_{0},$ that is, $U_\lambda$ intertwines the projectors:
$U_\lambda B_{\Pi}^{\lambda} =  \left( \chi_{+}I\otimes P_{0} \right) U_\lambda .$
\end{enumerate}
\end{theorem}

The equivalence of operators (projectors) in the theorem above and everywhere below is understood as a one-to-one correspondence between these operators by means of an isometric isomorphism.

Note that in the special case $p=q=2$ the following result is known (see \cite{Grudsky-Karapetyants-Vasilevski-JOpTh-2004} and also \cite{Vasilevski-book}).

\begin{theorem} (\cite{Grudsky-Karapetyants-Vasilevski-JOpTh-2004}, \cite{Vasilevski-book})\label{th:parabolic-decomp} Let $\lambda\in(-1,+\infty).$ The unitary operator $U_{\lambda}=U_{2,\lambda}U_{1,\lambda}$ is an isometric isomorphism of the space $L_{\lambda}^{2}(\Pi)$ onto
$L^{2}(\R,L^{2}(\R_{+})),$ such that
\begin{enumerate}
\item The Bergman space $\Aa_{\lambda}^{2}(\Pi)$ is mapped onto
\[
L^{2}(\R_{+},L^{2}_{0}(\R_{+})),
\]
\item The Bergman projection $B_{\Pi}^{\lambda}$ is unitary equivalent to the projection $\chi_{+}I\otimes P_{0}.$
\end{enumerate}
\end{theorem}
Let us introduce the isometric embedding
\[
R_{0}:L^{q}(\R_{+})\longrightarrow L^{q,p}(\Pi)
\]
as follows
\[
(R_{0}f)(x,y)=\chi_{+}(x)f(x)\ell_{0,p}(y), \,\,\, y>0.
\]
Here the function $f$ continues by zero for $x<0.$ The image of $R_{0}$ obviously coincides with the space $L^{q}(\R_{+})\otimes L^{p}_{0}(\R_{+})\, (=L^q(\R_+,L_0^p(\R_+)) ).$

The left-inverse operator
\[
R_{0}^{-1}:L^{q,p}(\Pi)\longrightarrow L^{q}(\R_{+})
\]
acts according to the formula
\begin{equation*}
(R_{0}^{-1}\varphi )(x)=\chi_{+}(x)\int_{\R_{+}}\varphi(x,\eta)\ell_{0,p}^{p-1}
(\eta)d\eta.
\end{equation*}
\begin{lemma}
We have:
\begin{eqnarray*}
R_0^{-1}R_0&=&I:L^q(\R_+) \longrightarrow L^q(\R_+),\\ R_0R_0^{-1}&\doteq&B_p:L^{q,p}(\Pi)\longrightarrow
L^q(\R_+,L_0^p(\R_+)).
\end{eqnarray*}
\end{lemma}
\begin{proof}
Indeed, for $f\in L^{q}(\R_{+}),$ $x>0$ we obtain:
\begin{eqnarray*}
(R_0^{-1}R_0f)(x,y) &=&R_0^{-1}\chi_+(x)f(x)e^{-\frac{y}{p}}\\&=& \int_{\R_+}\chi_+(x)f(x)\; e^{-\frac{\eta}{p}}\;\ell_{0,p}^{p-1}(\eta)d\eta\\&=&\chi_+(x)f(x)\int_{\R_+}e^{-\frac{\eta}{p}}\;e^{-\frac{\eta (p-1)}{p}}d\eta\\&=& \chi_+(x)f(x)\int_{\R_+}e^{-\eta}d\eta=\chi_+(x)f(x)=f(x).
\end{eqnarray*}
On the other hand, for $f=f(x,y)\in L^{q,p}(\Pi)$ the following equalities hold:
\begin{eqnarray*}
R_0R_0^{-1}f(x,y)&=&R_0 \chi_+(x)\int_{\R_+}f(x,\eta)e^{-\frac{\eta(p-1)}{p}}d\eta\\&=&
\chi_+(x)\ell_{0,p}(y)\int_{\R_+}f(x,\eta)\;e^{-\frac{\eta(p-1)}{p}}d\eta=(B_pf)(x,y).
\end{eqnarray*}
The convergence of the integrals written above is obvious for $p=1,$ and for $p>1$ it is justified by using Hölder's inequality:
\begin{eqnarray*}
\left|\int_{\R_+}f(x,\eta)e^{-\frac{\eta(p-1)}{p}}d\eta\right|&\leq&\left(\int_{\R_+}\left
|f(x,\eta)\right|^pd\eta\right)^{\frac{1}{p}} \left(\int_{\R_+}e^{-\eta}d\eta\right)^\frac{1}{p'}\\&=&
\parallel f(x,\cdot)\parallel_{L^p(\R_+)},\,\,\,\,\, \frac{1}{p}+\frac{1}{p'}=1.
\end{eqnarray*}
This finishes the proof.
\end{proof}

The operator $R_{\lambda}=R_{0}^{-1}U_{\lambda}$ maps the space ${\mathcal{L}}^{q,p}({\Pi})$ onto $L^ {q}(\R_{+}),$ and its restriction \begin{equation*} R_{\lambda}\bigg{|}_{\Aa_{\lambda}^{q,p}(\Pi)}: \Aa_{\lambda}^{q,p}(\Pi)\longrightarrow L^{q}(\R_{+}) \end{equation*} is an isometric isomorphism.

The left - inverse operator
\[
R_{\lambda}^{-1}=U_{\lambda}^{-1}R_{0}:L^{q}(\R_{+})\longrightarrow \Aa_{\lambda} ^{q,p}(\Pi)\subset {\mathcal{L}}_{\lambda}^{q,p}(\Pi)
\]
is an isometric isomorphism of the space $L^{q}(\R_{ +})$ on
$\Aa_{\lambda}^{q,p}(\Pi).$

\begin{remark}
We have:
\begin{eqnarray*}
R_{\lambda}R_{\lambda}^{-1} &=&I:L^q(\R_+)\longrightarrow L^q(\R_+),\\ R_{\lambda}^{-1}R_{\lambda}&=&B_{\Pi}^{\lambda}:{\mathcal{L}}_{\lambda}^{q,p}(\Pi)\longrightarrow \Aa_{\lambda}^{q,p}(\Pi).
\end{eqnarray*}
\end{remark}

Let us provide integral representations for the operators $R_{\lambda}, R_{\lambda}^{-1}.$ The following theorem is one of the main results of this paper.

\begin{theorem}(Paley-Wiener-type theorem).
The isometric isomorphism
\begin{equation*}
R_{\lambda}^{-1}=U_{\lambda}^{-1}R_{0}: L^{q}(\R_{+})\longrightarrow \Aa_{\lambda}^{q,p}(\Pi)
\end{equation*}
is defined as
\begin{equation}
\label{eq: Rr2}(R_{\lambda}^{-1}\varphi)(z)=\frac{p^{\frac{\lambda +1}{p}}}{\sqrt{2}\sqrt[p]{2^{\lambda}\Gamma(\lambda+2)}}\int\limits_{\R_+}\varphi(\xi)\xi^{\frac{\lambda +1}{p}}e^{i\xi z}d\xi,\ \ \varphi \in L^{q}(\R_{+}).
\end{equation}
The inverse isomorphism
\begin{equation*}
R_{\lambda}=R_{0}^{-1}U_{\lambda}:
\Aa_{\lambda}^{q,p}(\Pi)
\longrightarrow
L^{q}(\R_{+})
\end{equation*}
has the form:
\begin{eqnarray}\label{eq: R2}
&&(R_{\lambda}f)(x)=\chi_{+}(x)\frac{1}{\sqrt{2}}\;\theta_{\lambda,p}^{p-1}(x) \int\limits_{\prod}f(\omega)e^{-ix\overline{\omega}}e^{-x(p-2)\scriptsize{\re}\omega}
d\mu_{\lambda}(\omega)
\\ &&=
\nonumber \chi_{+}(x)\frac{\lambda+1}{\sqrt{2}}\;\theta_{\lambda,p}^{p-1}
(x)\ \int_{\R}\int_{\R_+}f(\xi+i\eta)e^{-ix(\xi-i\eta)}e^{-x(p-2)\xi}\frac{1}{\pi}
(2\eta)^{\lambda}d\eta d\xi,
\end{eqnarray}
where $f\in \Aa_{\lambda}^{q,p}(\Pi).$
\end{theorem}
\begin{proof}
Let us first show that for any function $\varphi\in L^{q}(\R_{+})$ the integral (\ref{eq: Rr2}) converges absolutely and defines an analytic function on $\Pi$. Indeed, taking into account the Hölder inequality, we have
\begin{eqnarray*}
&&\int\limits_{\R_+}\left|\varphi(\xi)\xi ^{\frac{\lambda +1}{p}}e^{i\xi z}\right|d\xi = \int\limits_{\R_+}\left|\varphi(\xi)\right|\xi^{\frac{\lambda +1}{p}}e^{-y\xi}d\xi\\&&\leq\left(\int_{\R_+}e^{-q'y\xi}\xi^{q'\frac{\lambda+1}{p}}
d\xi\right)^{\frac{1}{q'}}\left(\int_{\R_+}|\varphi(\xi)|^qd\xi\right)^{\frac{1}{q}}<\infty.
\end{eqnarray*}
The analyticity of the function on the right-hand side of (\ref{eq: Rr2}) follows from the fact that the formally differentiated integral in (\ref{eq: Rr2}) remains absolutely convergent uniformly in $z$ in some neighborhood of each point $z.$

Further, the following equalities hold:
\begin{eqnarray*}
R_{\lambda}^{-1}\varphi(x)&&=U_{\lambda}^{-1}R_0\varphi(x)=U_{1,\lambda}^{-1}
U_{2,\lambda}^{-1}R_0\varphi(x)\\&&= U_{1,\lambda}^{-1}U_{2,\lambda} ^{-1}\chi_+(x)\varphi(x)e^{-\frac{y}{p}}\\&&=U_{1,\lambda}^{-1}\theta_{\lambda,p}(x)
e^{-\frac{\psi_{\lambda}(\mid x\mid,y )}{p}-\mid x\mid y}\chi_+(x)\varphi(x)e^{\frac{\psi_{\lambda}(\mid x\mid,y )}{p}}\\&&=U_{1,\lambda}^{-1}\theta_{\lambda,p}(x)e^{-\mid x\mid y }\chi_+(x)\varphi(x)\\&&=\sqrt{\pi}\frac{1}{\sqrt{2\pi}}\int\limits_{\R}\theta_{\lambda,p}(\xi)
e^{-\mid \xi \mid y}\chi_+(\xi)\varphi(\xi)e^{ix\xi}d\xi\\&&=\frac{1}{\sqrt{2}}\int\limits_{\R_+}\left(\frac{\frac{p^{\lambda+1}}{2^{\lambda}}\xi^{\lambda+1}}{\Gamma(\lambda+2)}\right)^{\frac{1}{p}}e^{-\xi y }\varphi(\xi)e^{ix\xi}d\xi\\&&=\frac{p^{\frac{\lambda+1}{p}}}{\sqrt{2}\sqrt[p]{2^{\lambda}\Gamma(\lambda+2)}}\int\limits_{\R_+}\xi^{\frac{\lambda+1}{p}}\varphi(\xi)e^{i\xi z}d\xi,\,\,\,
z=x+iy.
\end{eqnarray*}
This gives us (\ref{eq: Rr2}).

Let $f=f(x+iy)\in\Aa_{\lambda}^{q,p}(\Pi),$ then the following equalities are satisfied:
\begin{eqnarray*}
&&R_{\lambda}f(x+iy)=R_0^{-1}U_{\lambda}f(x+iy)=R_0^{-1}U_{2,\lambda}U_{1,\lambda}f(x+iy)\\&=&
R_0^{-1}U_{2,\lambda}\frac{1}{\sqrt{\pi}} \frac{1}{\sqrt{2\pi}}\int_{\R}f(\xi+iy)e^{-ix\xi}d\xi\\&=&R_0^{-1}\frac{1}{\theta_{\lambda,p}(|x|)}\;e^{-\frac{y}{p}+|x|\beta_{\lambda}(|x|,y)}\frac{1}{\pi\sqrt{2}}
\int\limits_{\R} f(\xi+i\beta_{\lambda}(|x|,y))e^{-ix\xi}d\xi\\&=&\frac{\chi_{+}(x)}{\theta_{\lambda,p}(|x|)}\\&\times&
\int\limits_{\R_+}\left(
e^{-\frac{y}{p}+|x|\beta_{\lambda}(|x|,y)} \frac{1}{\pi\sqrt{2}}
\int\limits_{\R} f(\xi+i\beta_{\lambda}(|x|,y))e^{-ix\xi}d\xi\right)e^{-y\frac{p-1}{p}}
dy.
\end{eqnarray*}
In the last expression we will change the variables $y=\psi_{\lambda}(|x|,\eta),$ then we will have:
\begin{eqnarray*}
&&R_{\lambda}f(x+iy)=\frac{\chi_{+}(x)}{\sqrt{2}\pi\theta_{\lambda,p}(x)}
\\&&\times\int\limits_{\R_+}\left(e^{-\frac{\psi_{\lambda}(x,\eta)}{p}+x\eta}
e^{-\psi_{\lambda}(x,\eta)\frac{p-1}{p}} \int\limits_{\R} f(\xi+i\eta)e^{-ix\xi}d\xi\right)\frac{\partial}{\partial\eta}\psi_{\lambda}(x,\eta)
d\eta\\&&=\frac{\chi_{+}(x)}{\theta_{\lambda,p}(x)}\frac{\theta_{\lambda,p}^{p}(x)}{\sqrt{2}\pi}\int\limits_{\R_+}
e^{x\eta}e^{-px\eta}(\lambda+1)(2\eta)^ {\lambda} \int\limits_{\R} f(\xi+i\eta)e^{-ix\xi}d\xi
d\eta \\&&=\chi_+(x)\frac{\theta_{\lambda,p}^{p-1}(x)}{\sqrt{2}}\int\limits_{\R_+}\int\limits_{\R}
f(\xi+i\eta)e^{-ix\xi} e^{x\eta-px\eta}\frac{1}{\pi}(\lambda+1)(2\eta)^ {\lambda}
d\xi d\eta \\&&=\chi_+(x)\frac{\theta_{\lambda,p}^{p-1}(x)}{\sqrt{2}}\int\limits_{\R_+}\int\limits_{\R}
f(\xi+i\eta)e^{-ix(\xi-i\eta)} e^{-x\eta(p-2)}d\mu_{\lambda}(\xi,\eta)\\&&=\chi_+(x)\frac{1}{\sqrt{2}}\;\theta_{\lambda,p}^{p-1}(x)\int\limits_{\prod}f(\omega)e^{-ix\overline{\omega}}e^{-x(p-2){\scriptsize\re}\omega}d\mu_{\lambda}(\omega),
\end{eqnarray*}
where
$\omega=\xi+i\eta,\;\; d\mu_{\lambda}(\omega)=\frac{1}{\pi}(\lambda+1)(2\eta)^{\lambda}d\xi d\eta.$ This finishes the proof.
\end{proof}

As a corollary, we will provide a description of the functions from the class $\Aa_{\lambda}^{q,p}(\Pi).$

\begin{theorem}
Let $1\leqslant p, q<\infty,$ $\lambda>-1,$
$f\in\Aa_{\lambda}^{q,p}(\Pi).$ Then
$$\|f\|_{{\mathcal{L}}_{\lambda}^{q,p}(\Pi)}=\|R_{\lambda}f\|_{L^q(\mathbb{R}_+)}.$$
\end{theorem}

\begin{remark}Note that in the definition of the operator $R_{\lambda}$ from (\ref{eq: R2}) there is a factor (a weight) $x^{\frac{\lambda+1}{q}}$, see the definition of the function $\theta_{\lambda,p}$. It is determined from the relation (\ref{measurable_test_Lp}) due to the choice of a specific space $X(\mathbb{R}_+) = L^p(\mathbb{R}_+,\nu_\lambda).$ Having chosen another space $X(\mathbb{R}_+)$ (for example, Morrey, Orlicz, etc.), we naturally obtain another weight, and thus we will obtain another new space of analytic functions of the mixed-Fourier-norm type. Such spaces may differ significantly from each other in properties. This is a question for further research.
\end{remark}

\begin{remark}
Consider the definition of the operator $R^{-1}_{\lambda}$ from (\ref{eq: Rr2}) in the case $p=q=2.$ Put
$\psi(\xi)=\varphi(\xi)\xi^{\frac{\lambda+1}{2}}.$
Then obviously $\psi\in L^{2}(\R_{+},\xi^{-\lambda-1}),$ and we obtain an isometric isomorphism
\begin{equation}
L^{2}(\R_{+},\xi^{-\lambda-1}) \ni\psi \longleftrightarrow F\in \Aa_{\lambda}^{2}(\Pi),
\end{equation}
defined by the equality
\begin{equation}
\label{eq: Rra2a} F(z)= \frac{1}{\sqrt{\Gamma(\lambda+2)}}\int\limits_{\R_+}\psi(\xi)e^{i\xi z}d\xi,\ \ \psi \in L^{2}(\R_{+},\xi^{-\lambda-1}).
\end{equation} In fact, this is the classical Paley-Wiener theorem for the Bergman space $\Aa_{\lambda}^{2}(\Pi),$ see, for example, \cite{DurenGGMR}.
\end{remark}

\section{Toeplitz operators with symbols depending on $y=\im z$ in
$\Aa^{q,p}_{\lambda}(\Pi).$}
For a function $a=a(z)\in L^{1}_{\lambda}(\Pi)$, the Toeplitz operator $T_{a}^{(\lambda)}$ with symbol $a$, which is not necessarily bounded but defined on a dense set in
$\Aa^{2}_{\lambda} (\Pi)$, has the form:
\begin{equation*}
T_{a}^{(\lambda)}: f\in \Aa^{2}_{\lambda}(\Pi) \longmapsto B_{\Pi}^{(\lambda)}
a f \in \Aa^{2}_{\lambda}(\Pi).
\end{equation*}

We denote by $L^{1}_{\exp}(\R_{+})$ the class of functions $a=a(y),$ satisfying
\[
a(y)e^{-\varepsilon y}\in L^{1}(\R_{+})
\]
for any $\varepsilon >0.$ Let the function $\gamma_{a,\lambda}(x)$ be defined by the equality:
\begin{eqnarray}
\label{eq: ga1} \gamma_{a,\lambda}(x)=\frac{x^{\lambda+1}}{\Gamma(\lambda+1)}\int_0^{\infty}a\left(t/p\right)
e^{-tx}\: t^{\lambda}dt, \; x\in\R_{+} .
\end{eqnarray}

Here and in what follows we assume that the  symbol $a=a(y)$ of the Toeplitz operator belongs to $L^{1}_{\exp}(\R_{+})$ with the additional condition that for $-1<\lambda<0$ the function $a(y)y^{\lambda}$ is integrable in a neighborhood of zero.
\begin{theorem}
The Toeplitz operator $T_{a}^{(\lambda)}$ with symbol $a=a(y)$, acting in $\Aa_{\lambda}^{q,p}(\Pi),$ via isometric isomorphisms $R_{\lambda}, \; R_{\lambda}^{-1}$ is equivalent to the multiplication operator $$\gamma_{a,\lambda}I=R_{\lambda}T_{a}^{(\lambda)}R_{\lambda}^{-1},$$ acting in $L^{q}(\R_{+}).$ The function $\gamma_{a,\lambda}$ is given by the equality (\ref{eq: ga1}), and the operators $R_{\lambda}, \; R_{\lambda}^{-1}$ have the form (\ref{eq: Rr2}), (\ref{eq: R2}).
\end{theorem}
\begin{proof}Direct calculations show:
\begin{eqnarray*}
&&R_{\lambda}T_a^{(\lambda)}R_{\lambda}^{-1}= R_{\lambda}a(y)R_{\lambda}^{-1}=
R_0^{-1}U_{2,\lambda} U_{1,\lambda}a(y)U^{-1}_{1,\lambda}U_{2,\lambda}^{-1}R_0
\\&&=R_0^{-1}U_{2,\lambda} F_x a(y)F^{-1}_{x}U^{-1}_{2,\lambda}R_0=R_0^{-1}U_{2,\lambda}\{a(y)\}U^{-1}_{2,\lambda}R_0\\&&
= R_0^{-1}a(\beta_{\lambda}(|x|,y))R_0.
\end{eqnarray*}
Further, for any function $\varphi\in L^{q}(\R_{+})$ we have:
\begin{eqnarray*}
&&R_0^{-1}a(\beta_{\lambda}(|x|,y))R_0\varphi(x)=R_0^{-1}\chi_+(x)\;a(\beta_{\lambda}(|x|,y))
e^{-\frac{y}{p}}\varphi(x)\\&&=\chi_+(x)\int_{\R+}a(\beta_{\lambda}(|x|,y))e^{-\frac{y}{p}}\varphi(x)e^{-\frac{(p-1)y}{p}}dy\\&&=\chi_+(x)\varphi(x)\int_{\R+}a(\beta_{\lambda}(|x|,y))e^{-y}dy.
\end{eqnarray*}
Make the change of variables as follows
\[
y=\psi_{\lambda}(x,\tau), \; \beta_{\lambda}(|x|,\psi_{\lambda}(x,\tau))=\tau,
\]
then, taking into account the equality (\ref{eq: exp 2}), we obtain:
\begin{eqnarray*}
&&R_0^{-1}a(\beta_{\lambda}(|x|,y))R_0\varphi(x)\\&&=-\chi_+(x)\varphi(x)\int_{\R_+}a(\tau)d\left[
\frac{x^{\lambda+1}}{\Gamma(\lambda+1)}\int_{\tau}^{\infty}e^{-p\zeta x}(p\zeta)^{\lambda}d(p\zeta)\right]\\&&=\chi_+(x)\varphi(x)\frac{x^{\lambda+1}}{\Gamma(\lambda+1)}
\int_{\R_+}a(\tau)e^{-p\tau x}(p\tau)^{\lambda}d(p\tau)\\&&=\chi_+(x)\varphi(x)\frac{x^{\lambda+1}}{\Gamma(\lambda+1)}
\int_{\R_+}a(t/p)e^{-t x}t^{\lambda}dt=\gamma_{a,\lambda}(x)\varphi(x),
\end{eqnarray*}
where $t=p\zeta.$
\end{proof}
\begin{corollary}
The Toeplitz operator $T_{a}^{(\lambda)}$ with symbol $a=a(y)$ is bounded on $\Aa_{\lambda}^{q,p}(\Pi)$ if and only if the corresponding function $\gamma_{a,\lambda}$ is bounded. The spectrum of the bounded Toeplitz operator $T_{a}^{(\lambda)}$ is the set $\spec T_{a}^{(\lambda)}=\overline{\Range\gamma_{a,\lambda}}.$
\end{corollary}

\vspace{1cm}

\noindent {\bf Funding information.}
The work is supported by the Ministry of Education and Science of Russia, agreement No. 075-02-2024-1427. The first author is supported by the FWO Odysseus 1 grant G.0H94.18N: Analysis and Partial Differential Equations and by the Methusalem programme of the Ghent University Special Research Fund (BOF) (Grant number 01M01021).

\noindent
{\bf{Data availability and conflict of interest.}}
The authors confirm that all data generated or analyzed during this study are included in this article.
This work does not have any conflicts of interest.

\addcontentsline{toc}{section}{

\vspace{1cm}

\noindent
{\bf Zhirayr Avetisyan},
Ghent University, Ghent, Belgium, jirayrag@gmail.com

\noindent
{\bf Alexey Karapetyants},
Institute of Mathematics, Mechanics and Computer Sciences and Regional Mathematical Center of Southern Federal University, Rostov-on-Don, Russia, karapetyants@gmail.com

\noindent
{\bf Irina Smirnova},
Don State Technical University, Rostov-on-Don, Russia, irinasmiy@yandex.ru

\end{document}